\documentclass[final,leqno]{siamltex1213}

\usepackage{cite}
\usepackage{makeidx}
\usepackage{enumitem}
\usepackage{amssymb, amsmath}
\makeindex

\hypersetup{
    colorlinks=true,        
    linkcolor=red,          
    citecolor=green,        
    filecolor=magenta,      
    urlcolor=cyan           
}

\definecolor{rblue}{rgb}{.255,.41,.884} 

\newcommand{\textdef}[1]{\textit{#1}\index{#1}}

\newtheorem{example}[theorem]{Example}

\newtheorem{lem}[theorem]{Lemma}
\newtheorem{thm}[theorem]{Theorem}

\newtheorem{rem}[theorem]{Remark}
\newtheorem{remark}[theorem]{Remark}

\def\R{\mathbb{R}}
\def\E{\mathbb{E}}
\def\Y{\mathbb{Y}}

\def\Ss{\mathcal{S}}

\newcommand{\Sn}{\mathcal{S}^n}
\newcommand{\Snp}{\mathcal{S}^n_+}
\newcommand{\Srp}{\mathcal{S}^r_+}

\newcommand{\Sr}{\mathcal{S}^r}

\newcommand{\EE}{{\mathcal E} }

\newcommand{\KK}{{\mathcal K} }

\newcommand{\PP}{{\mathcal P} }

\DeclareMathOperator{\trace}{{trace}}

\DeclareMathOperator{\Diag}{{Diag}}
\DeclareMathOperator{\range}{{rge}}
\DeclareMathOperator{\kernel}{{ker}}
\DeclareMathOperator{\face}{{face}}
\DeclareMathOperator{\bnd}{{bnd}}
\DeclareMathOperator{\cl}{{cl}}
\DeclareMathOperator{\embdim}{{embdim}}

\newcommand{\MM}{{\mathcal M}}

\DeclareMathOperator{\interior}{int}
\DeclareMathOperator{\linspan}{span}

\DeclareMathOperator{\ri}{{ri}}
\DeclareMathOperator{\dir}{{dir}}

\DeclareMathOperator{\tcone}{{tcone}}
\DeclareMathOperator{\dist}{{dist}}

\newcommand{\commentout}[1]{}
\newcommand{\co}[1]{}

\newcommand{\sym}[1]{{\mathcal S}^{#1}}
\newcommand{\psd}[1]{{\mathcal S}_+^{#1}}

\newcommand{\rad}[1]{\mathbb{R}^{#1}}

\makeatletter
\renewcommand*\env@matrix[1][c]{\hskip -\arraycolsep
  \let\@ifnextchar\new@ifnextchar
  \array{*\c@MaxMatrixCols #1}}
\makeatother

\title{
Coordinate shadows of semi-definite
    and Euclidean distance matrices
} 

\author{
\href{http://math.washington.edu/\~ddrusv}{Dmitriy Drusvyatskiy}\thanks{Department of Mathematics, University of Washington, Seattle, WA 98195-4350; \url{www.math.washington.edu/\~ddrusv}; Research supported by AFOSR.}
\and
\href{http://www.unc.edu/~pataki/}{G\'{a}bor Pataki}\thanks{Department of Statistics and Operations Research, University of
North Carolina at Chapel Hill; \url{www.unc.edu/\~pataki}.}
 \and 
\href{http://www.math.uwaterloo.ca/~hwolkowi/}{Henry
Wolkowicz}\thanks{Department of Combinatorics and Optimization,
University of Waterloo, Waterloo, Ontario, Canada N2L 3G1; \url{www.math.uwaterloo.ca/\~hwolkowi}; Research supported by NSERC and by AFOSR.}
}

\begin{document}
\maketitle
\slugger{mms}{xxxx}{xx}{x}{x--x}

\begin{abstract}
We consider the projected semi-definite and Euclidean distance cones onto a subset of the matrix entries. These two sets are precisely the input data defining feasible semi-definite and
Euclidean distance completion problems. We classify 
when these sets are closed, and use the boundary structure of these two sets to elucidate the Krislock-Wolkowicz facial reduction algorithm. In particular, we show that under a chordality assumption, the ``minimal cones'' of these problems admit combinatorial characterizations. As a byproduct, we record a striking relationship between the complexity of the general facial reduction algorithm (singularity degree) and facial exposedness of conic images under a linear mapping.
\end{abstract}



\pagestyle{myheadings}
\thispagestyle{plain}
\markboth{COORDINATE SHADOWS}{SEMI-DEFINITE AND EUCLIDEAN DISTANCE
MATRICES}

\section{Introduction}
To motivate the discussion, consider an undirected graph $G$ with vertex set $V=\{1,\ldots, n\}$ and edge set $E\subset \{ij: i\leq j\}$. The classical {\em semi-definite (PSD) completion problem} asks whether given a data vector $a$ indexed by $E$, there exists an $n\times n$ positive semi-definite matrix $X$ completing $a$, meaning $X_{ij}=a_{ij}$ for all $ij\in E$. Similarly, the {\em Euclidean distance (EDM) completion problem} asks whether given such a data vector, there exists a Euclidean distance matrix completing it. For a survey of these two problems, see for example \cite{MR1607310, MR1778240, Floudas:2001aa, Netzer:2012aa}. The semi-definite and Euclidean distance completion problems are often mentioned in the same light due to a number of parallel results; see e.g. \cite{MR1491595}. 
Here, we consider a related construction: projections of the PSD cone $\Snp$ and the EDM cone $\EE^n$ onto matrix entries indexed by $E$. These ``coordinate shadows'', denoted by $\PP(\Snp)$ and $\PP(\EE^n)$, respectively, appear naturally: they are precisely the sets of data vectors that render the corresponding completion problems feasible. 
We note that these sets are interesting types of ``spectrahedral shadows'' --- 
an area of intensive research in recent years. For a representative sample 
of recent papers on spectrahedral shadows, we refer  to \cite{Netzer:2012aa, MR3062007, MR2515796, MR2533752,MR1404832} and references therein.

In this paper, our goal is twofold: (1) we will highlight the geometry of the two sets $\PP(\Snp)$ and $\PP(\EE^n)$, and (2) illustrate how such geometric considerations yield a much simplified and transparent analysis of an EDM completion algorithm proposed in \cite{Nathan_Henry}.
To this end, we begin by asking a basic question: 
\smallskip
\begin{center}
Under what conditions are the coordinate shadows $\PP(\Snp)$ and  $\PP(\EE^n)$ closed? 
\end{center}
\smallskip
This question sits in a broader context still of deciding if a linear image of a closed convex set is itself closed ---  a thoroughly studied topic due to its fundamental connection to constraint qualifications and strong duality in convex optimization; see e.g. \cite{con:70, MR85e:90057, Duff:56, Pataki:07, GPat:11} and references therein.  
In contrast to the general setting, a complete answer to this question in our circumstances is easy to obtain. 
An elementary argument\footnote{\label{foot:proofs}The elementary proofs
of Theorem~\ref{thm:maincompl_fixed} and \ref{thm:close_EDM_fixed} were
suggested to us by an anonymous referee, for which we are very grateful.
Our original reasoning, based on more general considerations, appear in
Subsection~\ref{sec:subsec_orig}.} shows that $\PP(\EE^n)$ is always
closed, whereas $\PP(\Snp)$ is closed if, and only if, the set of vertices attached to self-loops $L=\{i\in V: ii\in E\}$ is disconnected from its complement $L^c$ (Theorems~\ref{thm:maincompl_fixed}, \ref{thm:close_EDM_fixed}). Moreover, whenever there is an edge joining $L$ and $L^c$, one can with ease exhibit vectors lying in the closure of $\PP(\Snp)$,  but not in the set $\PP(\Snp)$ itself, thereby certifying that $\PP(\Snp)$ is not closed.

To illustrate the algorithmic significance of the coordinate shadows $\PP(\Snp)$ and $\PP(\EE^n)$, consider first the feasible region of the PSD completion problem:
$$\{X\in\Snp: X_{ij}=a_{ij} \textrm{ for }ij\in E \}.$$
For this set to be non-empty, the data vector $a\in\R^{E}$ must be a partial PSD matrix, meaning all of its principal submatrices are positive semi-definite. This, however, does not alone guarantee the inclusion $a\in\PP(\Snp)$, unless the restriction of $G$ to $L$ is chordal and $L$ is disconnected from $L^c$ (Theorem~\ref{thm:psd_comp},\cite[Theorem 7]{GrJoSaWo:84}). On the other hand, the authors of \cite{Nathan_Henry} noticed that even if the feasible set is nonempty, the Slater condition (i.e. existence of a positive definite completion) will often fail: small perturbations to any specified principal submatrix of $a$ having deficient rank can yield the semi-definite completion problem infeasible. In other words, in this case the partial matrix $a$ lies on the boundary of $\PP(\Snp)$  --- the focus of this paper.  
An entirely analogous situation occurs for EDM completions $$\{X\in\EE^n: X_{ij}=a_{ij} \textrm{ for }ij\in E \},$$
with the rank of each principal submatrix of $a\in\R^E$ replaced by its ``embedding dimension''. 
In \cite{Nathan_Henry}, the authors propose a preprocessing strategy utilizing the cliques in the graph $G$ to systematically decrease the size of the EDM completion problem. Roughly speaking, the authors use each clique to find a face of the EDM cone containing the entire feasible region, and then iteratively intersect such faces. The numerical results in \cite{Nathan_Henry} were impressive. In the current work, we provide a much simplified and transparent geometric argument behind their algorithmic idea, with the boundary of $\PP(\EE^n)$ playing a key role. As a result, $(a)$ we put their techniques in a broader setting unifying the PSD and EDM cases, and $(b)$ the techniques developed here naturally lead to a robust variant of the method for {\em noisy (inexact)} EDM completion problems \cite{sensor_pre} -- a more realistic setting. In particular, we show that when $G$ is chordal and all cliques are considered, the preprocessing technique discovers the minimal face of $\EE^n$ (respectively $\Snp$) containing the feasible region; see Theorems~\ref{thm:cliquesuffPSD} and \ref{thm:fin_term}. This in part explains the observed success of the method \cite{Nathan_Henry}. 
Thus in contrast to general semi-definite programming, the minimal face of the PSD cone containing the feasible region of the PSD completion problem under a chordality assumption (one of the simplest semi-definite programming problems) admits a purely combinatorial description.

As a byproduct, we record a striking relationship between the complexity of the general facial reduction algorithm (singularity degree) and facial exposedness of conic images under a linear mapping; see Theorem~\ref{thm:face_red}.
To the best of our knowledge, this basic relationship has either gone unnoticed or has mostly been forgotten.

The outline of the manuscript is as follows.
In Section \ref{sect:prel} we record basic results on convex
geometry and PSD and EDM completions.
In Section~\ref{sect:closedness}, we consider when the coordinate shadows $\PP(\Snp)$ and $\PP(\EE^n)$ are closed, while in Section \ref{sect:bndry} we discuss the aforementioned clique facial reduction strategy.

\section{Preliminaries}
\label{sect:prel}
\subsection{Basic elements of convex geometry}
We begin with some notation, following closely the classical 
text~\cite{con:70}. Consider a Euclidean space $\E$ with the inner product 
$\langle \cdot,\cdot \rangle$. The adjoint of a linear mapping 
$\MM\colon \E\to\Y$, between two Euclidean spaces $\E$ and $\Y$, 
is written as $\MM^{*}$, while the range and kernel of 
$\MM$ are denoted by $\range \MM$ and $\ker \MM$, respectively.
We denote the closure, boundary, interior, and relative interior of a set $Q$ in $\E$ by $\cl Q$, $\bnd Q$, $\interior Q$, and $\ri Q$, respectively.
Consider a convex cone $C$ in $\E$. 
The linear span and the orthogonal complement of the linear span of $C$ 
will be denoted by $\linspan C$ and $C^{\perp}$, respectively.
For a vector $v$, we let $v^\perp:=\{v\}^\perp$.
We associate with $C$ the \emph{nonnegative polar} cone
\[
C^*=\{y\in\E: \langle y,x\rangle \geq 0~ \textrm{ for all } x \in C\}.
\]
The second polar $(C^*)^{*}$ coincides with the original $C$
if, and only if, $C$ is closed.
A convex subset $F\subseteq C$ is a \textdef{face of $C$}, denoted 
$F \unlhd C$, if $F$ contains any line segment in $C$ whose relative interior intersects $F$. 
The \textdef{minimal face} containing a set $S\subseteq C$, denoted
$\face(S,C)$, is the intersection of all faces of $C$ containing $S$. When $S$ is itself a convex set, then $\face(S,C)$ is the smallest face of $C$ intersecting the relative interior of $S$.
A face $F$ of $C$ is an \textdef{exposed face} when there exists a 
vector $v\in C^*$ satisfying $F=C\cap \, v^{\perp}$. In this case, we say that $v$ {\em exposes} $F$.
The cone $C$ is \textdef{facially exposed} when all faces of $C$ are exposed.
In particular, the cones of positive semi-definite and Euclidean distance matrices, which we will focus on shortly, are facially exposed.
With any face $F\unlhd C$, we associate a face of the polar $C^*$, 
called the \textdef{conjugate face}
$F^{\triangle}:= C^*\cap F^\perp$. Equivalently, $F^{\triangle}$ is the 
face of $C^*$ exposed by any point $x\in \ri F$, that is 
$F^{\triangle}:= C^*\cap \,x^\perp$.
Thus, in particular, conjugate faces are always exposed. 
Not surprisingly then equality $(F^{\triangle})^{\triangle}=F$ holds if,
and only if, $F\unlhd C$ is exposed.

\subsection{Semi-definite and Euclidean distance matrices}\label{subsec:PSD_EDM}
We will focus on two particular realizations of the Euclidean space $\E$:  
the $n$-dimensional vector space $\R^n$ with a fixed basis and the induced 
\textdef{dot-product} $\langle \cdot,\cdot \rangle$ and the vector space of
$n\times n$ real symmetric matrices $\Sn$ with the 
\textdef{trace inner product} $\langle A,B\rangle :=\trace AB$.
The symbols $\R_{+}$ and $\R_{++}$ 
will stand for the non-negative orthant and its interior in $\R^n$, 
while $\Sn_{+}$ and $\Sn_{++}$ will stand for the cones of 
positive semi-definite and positive definite matrices in $\Sn$
(or \textdef{PSD} and \textdef{PD} for short), respectively. We let $e\in\R^n$ be the vector of all ones and for any vector $v\in\R^n$, the symbol $\Diag v$ will denote the $n\times n$ diagonal matrix with $v$ on the diagonal.

It is well-known that all faces of $\Sn_+$ are convex cones that can be expressed as
$$F=\left\{ U\begin{bmatrix} A & 0 \\ 0 & 0 \end{bmatrix}U^T: A\in\Sr_+
\right\},$$
for some orthogonal matrix $U$ and some integer $r=0,1,\ldots,n$.
Such a
face can equivalently be written as $F=\{X\in\Snp: \range X\subset\range
\overline{U}\}$, where $\overline{U}$ is formed from the first $r$
columns of $U$.
The conjugate face of such a face $F$ is then 
$$F^{\triangle}=\left\{ U\begin{bmatrix} 0 & 0 \\ 0 & A \end{bmatrix}
U^T: A\in\mathcal{S}^{n-r}_+ \right\}.$$
For any convex set $Q\subset\Snp$, the cone $\face(Q,\Snp)$ coincides with $\face(X,\Snp)$ where $X$ is any maximal rank matrix in $Q$.

A matrix $D\in\Sn$ is a {\em Euclidean distance matrix} (or EDM for short) if there 
exist $n$ points $p_i$ (for $i=1,\ldots, n$) in some Euclidean space 
$\R^k$ satisfying $D_{ij}=\|p_i-p_j\|^2$, for all indices $i,j$. These points are then said to {\em realize} $D$. 
The smallest integer $k$ for which this realization of $D$ by $n$ points is 
possible is the \textdef{embedding dimension} of $D$ and will be denoted by
$\embdim D$. 
We let $\mathcal{E}^n$ be the set of $n\times n$ Euclidean distance matrices.
There is a close relationship between PSD and 
EDM matrices. Indeed $\mathcal{E}^n$ is a closed convex cone that is linearly 
isomorphic to $\mathcal{S}^{n-1}_+$. 
To state this precisely, consider the mapping
$$\KK : \Sn \rightarrow \Sn$$ 
defined by  
\begin{equation}
\label{eq:defK}
\KK(X)_{ij}:=X_{ii}+X_{jj}-2X_{ij}.
\end{equation}
Then the adjoint 
$\KK^*\colon \Sn\to\Sn$ is given by 
\begin{equation*}
\KK^*(D)=2(\Diag(De)-D)
\end{equation*}
and the equations
\begin{equation} 
\label{rangeK*}
\range \KK = \Ss_H, \qquad \range \KK^*=\Ss_c
\end{equation}
hold, where
\begin{equation}
\label{eq:centholl}
\Ss_c:= \{X \in \Sn: Xe=0\}; \qquad
\Ss_H:= \{D \in \Sn: \diag(D)=0\},
\end{equation}
are the \textdef{centered} and  
 \textdef{hollow matrices},
respectively. It is known that $\KK$ maps $\Snp$ onto $\EE^n$, and moreover the restricted mapping
\begin{equation}
\label{eq:isomorph}
\KK : \Ss_c \rightarrow \Ss_H \text{  is a linear isomorphism carrying } \Ss_c\cap\Snp \textrm{ onto }\EE^n. 
\end{equation}
In turn, it is easy to see that $\Ss_c\cap\Snp$ is a face of $\Snp$ isomorphic to $\Ss^{n-1}_+$; see the discussion after Lemma~\ref{lem:face_iso} for more details. 
These and other related results have appeared in a number of publications; see for example \cite{homwolkA:04,MR1366579,hwlt91,MR2166851, MR97h:15032,MR2549047,MR2890931,MR2653818,MR2357790}.

\subsection{Semi-definite and Euclidean distance completions}
The focus of the current work is on the PSD and EDM completion problems, see e.g.,~\cite[Chapter 49]{MR2279160}.
Throughout the rest of the manuscript, we fix an undirected graph $G=(V,E)$, with a 
vertex set $V=\{1,\ldots,n\}$ and an edge set $E\subset \{ij: 1\leq i\leq j\leq n\}$. Observe that we allow self-loops. These loops will play an important role in what follows, and hence we define $L$ to be the set of all vertices $i$ satisfying $ii\in E$, that is those vertices that are attached to a loop.

Any vector $a\in\R^{E}$ is called a {\em partial matrix}.
Define now the projection map $\mathcal{P}: \Sn \rightarrow \R^{E}$ by setting 
\[
  \PP(A) =  (A_{ij})_{ij\in E },
\]
that is $\mathcal{P}(A)$ is the vector of all the entries of $A$ indexed by $E$.
The adjoint map $\PP^*:\R^{E} \rightarrow \Sn$  is found by setting
\[
\left(\PP^*(y)\right)_{ij}= \left\{
\begin{array}{ll}
   y_{ij}, & \text{if } ij\in E\\
   0, & \text{otherwise},
\end{array}
  \right.
\]
for indices $i\leq j$.
Define also the \textdef{Laplacian operator} $\mathcal{L}\colon \R^{E}\to \Sn$
by setting 
$$\mathcal{L}(a):=\frac{1}{2}(\PP\circ  \KK)^{*}(a)=\Diag(\PP^*(a)e)-\PP^*(a).$$
Consider a partial matrix $a\in\R^{E}$ whose components are all strictly
positive. Classically then the Laplacian matrix $\mathcal{L}(a)$ is
positive semi-definite and moreover the kernel of $\mathcal{L}(a)$ is
only determined by the connectivity of the graph $G$; see for example
~\cite{BrualRys:91},\cite[Chapter 47]{MR2279160}. Consequently all
partial matrices with strictly positive weights define the same minimal face of
the positive semi-definite cone. In particular, when $G$ is connected,
we have the equalities  
\begin{equation}\label{eqn:lap}
\kernel \mathcal{L}(a)=\linspan\{e\} \qquad\textrm{ and }\qquad
\face(\mathcal{L}(a),\Snp)=\mathcal{S}_c \cap \Snp.
\end{equation}

A symmetric matrix $A\in\Sn$ is a \textdef{completion} of a partial
matrix $a\in \R^{E}$ if it satisfies $\mathcal{P}(A)=a$. 
We say that a completion $A\in \Sn$ of a partial matrix $a\in\R^{E}$ is a 
{\em PSD completion} if $A$ is a PSD matrix. 
Thus the image $\PP(\Sn_+)$ is the set of all partial matrices 
that are PSD completable. 
A partial matrix $a\in \R^{E}$ is a 
\textdef{partial PSD matrix} if all existing 
principal submatrices, defined by $a$, are PSD matrices. 
Finally we call $G$ itself a {\em PSD completable graph} if every partial PSD matrix $a\in\R^E$ is completable to a PSD matrix. \textdef{PD completions}, \textdef{partial PD matrices}, and \textdef{PD completable graphs}
are defined similarly.

We call a graph \emph{chordal} if any cycle of four or more 
nodes (vertices) has a chord, 
i.e.,~an edge exists joining any two nodes that are not 
adjacent in the cycle. Before we proceed, a few comments on
completability are in order. In \cite[Proposition 1]{GrJoSaWo:84}, the
authors claim that $G$ is PSD completable (PD respectively)  if, and
only if, the graph induced on $L$ by $G$ is PSD completable (PD respectively). In light of this, the authors then reduce all of their arguments to this induced subgraph. It is easy to see that the statement above does not hold for PSD completability (but is indeed valid for PD completability). 
 Consider for example the partial PSD matrix
 $\begin{bmatrix}
 0 &1\\
 1 & ?
 \end{bmatrix}$
 which is clearly not PSD completable. 
Taking this into account, the correct statement of their main result \cite[Theorem 7]{GrJoSaWo:84} is as follows. See also the discussion in \cite{Laurent:00}. 
\smallskip

\begin{theorem}[PSD completable matrices \& chordal graphs]\label{thm:psd_comp}{\hfill \\}
The following are true.
\begin{enumerate}
\item The graph $G$ is PD completable if, and only if, the graph induced by $G$ on $L$ is chordal.
\item The graph $G$ is PSD completable if, and only, if the graph induced by $G$ on $L$ is chordal and $L$ is disconnected from $L^c$.
\end{enumerate}
\end{theorem}
\smallskip

With regard to EDMs, we will always assume 
$L=\emptyset$ for the simple reason that the diagonal of 
an EDM is always fixed at zero.
With this in mind, we say that a completion $A\in \Sn$ of a partial 
matrix $a\in\R^{E}$ is an \textdef{EDM
completion} if $A$ is an EDM. 
Thus the image $\PP(\mathcal{E}^n)$ (or equivalently 
$\mathcal{L}^*(\Snp)$) is the set of all partial matrices 
that are EDM completable.
We say that a partial matrix $a\in\R^{E}$ is a \textdef{partial EDM} if any existing principal submatrix, defined by $a$, is an EDM. 
Finally we say that $G$ is an EDM completable graph if any partial EDM is completable to an EDM.
The following theorem is analogous to Theorem~\ref{thm:psd_comp}. For a proof, see~\cite{MR1321802}.
\smallskip
\begin{thm}[Euclidean distance completability \& chordal graphs]
\label{thm:compledm}{\hfill \\ } 
The graph $G$ is EDM completable if, and only if, $G$ is chordal.
\end{thm}

\section{Closedness of the projected PSD and EDM cones}
\label{sect:closedness}
We begin this section by characterizing when the projection of the 
PSD cone $\Snp$ onto some subentries is closed. 
To illustrate, consider the simplest setting $n=2$, namely 
\[
\mathcal{S}^2_+ =
 \left\{\begin{bmatrix} x  & y \\
y & z \end{bmatrix} : x\geq 0, z\geq 0, xz\geq y^2\right\}.
\]
Abusing notation slightly, one can easily verify:
$$\PP_{z}(\mathcal{S}^2_+)=\R_+ ,\qquad \PP_{y}(\mathcal{S}^2_+)=\R, \qquad \PP_{x,z}(\mathcal{S}^2_+)=\R^2_+.$$
Clearly all of these projected sets are closed. Projecting $\mathcal{S}^2_+$ onto a single row, on the other hand, yields a set that is not closed:
$$\PP_{x,y}(\mathcal{S}^2_+)=\PP_{z,y}(\mathcal{S}^2_+)=\{(0,0)\}\cup(\R_{++}\times\R).$$
In this case, the graph $G$ has two vertices and two 
edges, and in particular, there is an edge joining $L$ with $L^c$. As we will now see, this connectivity property is the only obstacle to $\PP(\Snp)$ being closed. The elementary proof of the following two theorems was suggested to us by one of the anonymous referees, for which we are very grateful. Our original arguments, based on more general principles, now appear in Subsection~\ref{sec:subsec_orig}, and may be of an independent interest. 
\smallskip

\begin{thm}[Closedness of the projected PSD cone] 
	\label{thm:maincompl_fixed} 
	The projected set $\PP(\psd{n})$ is closed if, and only if, the vertices in $L$ are disconnected from those in the complement $L^c$.
	Moreover, if the latter condition fails, then 
	for any edge $i^*j^*\in E$ joining  a vertex in $L$ with a
	vertex in $L^c$, any partial matrix $a\in\R^{E}$ satisfying
	\begin{align*}
	a_{i^*j^*} \neq 0 \quad&\textrm{ and }\quad 
	a_{ij}=0 \textrm{   for all } ~ij\in E\cap (L\times L),
	\end{align*}
	lies in $\big(\cl \PP(\Sn_+)\big)\setminus \PP(\Sn_+)$.
\end{thm}
\begin{proof}	
Without loss of generality, assume $L=\{1,\ldots,r\}$ for some integer $r\geq 0$.	
Suppose first that the vertices in $L$ are disconnected from those in the complement $L^c$. Let $a_i\in\R^E$  be a sequence in $\PP(\psd{n})$ converging to a partial matrix $a\in \R^E$. We may now write $a_i=\PP(X_i)$ for some matrices $X_i\in \psd{n}$. Denoting by $X_i[L]$ the restriction of $X_i$ to the $L\times L$ block, we deduce that the diagonal elements of $X_i[L]$ are bounded and therefore
the matrices $X_i[L]$ are themselves bounded.
Hence there exists a subsequence of $X_i[L]$ converging to some PSD matrix $X_L$. Let $Y\in \mathcal{S}^{|L^c|}$ now be any completion of the restriction of $a$ to $L^c$. Observe that for sufficiently large values $\lambda$ the matrix $Y+\lambda I$ is positive definite and hence $\begin{bmatrix} X_L  & 0 \\
	0 & Y+\lambda I \end{bmatrix}$
is a positive semi-definite completion of $a$.

Conversely, suppose that $L$ is not disconnected from $L^c$ and consider 
the vertices $i^*,j^*$ and a matrix $a$ specified in the statement of the theorem. Since the block 
$\begin{bmatrix} 0  & a_{i^*j^*} \\
a_{i^*j^*} & ? \end{bmatrix}$
is not PSD completable, clearly $a$ is not PSD completable. To see the inclusion $a\in \cl\PP(\Sn_+)$, consider the matrix $X:=\PP^*(a)+ \begin{bmatrix} \epsilon I  & 0 \\
	0 & \lambda I \end{bmatrix}$.
Using Schur's complement, we deduce that for any fixed $\epsilon$ there
exists a sufficiently large $\lambda$ such that $X$ is positive
definite. On the other hand, clearly $\PP(X)$ converges to $a$ as $\epsilon$ tends to zero. This completes the proof.
\end{proof}

\smallskip


\smallskip

\smallskip
We next consider projections of the EDM cone.
 
\smallskip
\begin{theorem}[Closedness of the projected EDM cone]\label{thm:close_EDM_fixed} 
	\hfill {\\ }
	The projected image $\PP(\mathcal{E}^n)$ is always closed.
\end{theorem}
\begin{proof}
First, we can assume without loss of generality that the graph $G$ is connected. To see this, let $G_i=(V_i, E_i)$ for $i=1,\ldots,l$ be the connected components of $G$.
Then one can easily verify that $\PP(\mathcal{E}^n)$ coincides with the Cartesian product $P_{E_1}(\mathcal{E}^{|V_1|})\times \ldots\times P_{E_l}(\mathcal{E}^{|V_l|})$. Thus if each image $\PP_{E_i}(\mathcal{E}^{|V_i|})$ is closed, then so is the product $\PP(\mathcal{E}^n)$.
We may therefore assume that $G$ is connected. Now suppose that for a sequence $D_i\in \EE^n$ the vectors $a_i=\PP(D_i)$ converge to some vector $a\in\R^E$. Let $x^i_1,\ldots,x^i_n$ be the point realizing the matrices $D_i$. Translating the points, we may assume that one of the points is the origin. Since $G$ is connected, all the points $x^i_j$ are bounded in norm. Passing to a subsequence, we obtain a collection of points realizing the matrix $a$.
\end{proof}


\subsection{Alternate proofs}\label{sec:subsec_orig}
Theorems~\ref{thm:maincompl_fixed} and \ref{thm:close_EDM_fixed} are part of a broader theme.
Indeed, a central (and classical) question in convex analysis is when a linear
image of a closed convex cone is itself closed. In a recent paper~\cite{Pataki:07}, the
author showed that there is a convenient
characterization for ``nice cones'' --- those cones $C$ for which $C^*+F^\perp$ is closed for all faces 
$F \unlhd C$ \cite{BW80,Pataki:07}. 
Reassuringly, most cones which we can efficiently 
optimize over are nice; see the discussion in~\cite{Pataki:07}.
For example, the cones of positive semi-definite and 
Euclidean distance matrices are nice. In this subsection, we show how the results of Theorems~\ref{thm:maincompl_fixed}, \ref{thm:close_EDM_fixed} can be recovered from the more general perspective; originally, the content of the aforementioned results were noticed exactly in this way. (We note that the results 
in~\cite{Pataki:07} provide {\em necessary} conditions for the linear image of 
any closed convex cone to be closed.)

To proceed, we need two standard notions:
the  
\textdef{cone of feasible directions} and the {\em tangent cone} of a convex cone $C$ at one of its points $x$, respectively, are the sets
\begin{align*}
\dir(x,C)  &:=  \left\{ 
v  : x + \epsilon v \in C \textrm{ for some } \epsilon >0 \right\},\\
\tcone(x,C) &:= \cl \dir (x,C).
\end{align*}
Both of the cones above can conveniently be described in terms of the minimal face $F:=\face(x,C)$ as follows (for details, see~\cite[Lemma~1]{GPat:11}):
\begin{equation*}
\label{eq:fourconessums}
\begin{array}{rclrcl}
\dir(x,C) = C+\linspan F \qquad\textrm{ and }\qquad \tcone(x,C) = (F^{\triangle})^*. 
\end{array}
\end{equation*}
The following theorem, originating in 
\cite[Theorem 1.1, Corollary 3.1]{Pataki:07} and~\cite[Theorem 3]{GPat:11}, provides a general framework for checking image closedness.

\smallskip
\begin{thm}
	[Image closedness of nice cones]
	\label{pataki-cl}
	Let $\MM:\E\rightarrow \Y$ be a linear transformation 
	between two Euclidean spaces $\E$ and $\Y$, and let $C\subseteq \Y$ be a
	nice, closed convex cone. Consider a point $x\in\ri(C\cap \,\range
	\MM)$. Then the following two statements are equivalent.
	\begin{enumerate}
		\item The image $\MM^*C^*$ is closed.
		\item\label{claim:2} The implication
		\begin{equation}
		\label{eqn:unlin} 
		v\in \tcone(x,C) \cap\, \range \MM \quad \Longrightarrow \quad v\in \dir(x,C)  \quad
		\textrm{holds}.
		\end{equation}
	\end{enumerate}
	Moreover, suppose that implication \eqref{eqn:unlin} fails and choose an arbitrary vector 
	$v\in \left(\tcone(x,C) \cap \range \MM\right)\setminus \dir(x,C)$. 
	Then for any point
	\begin{equation}
	\label{eq:av}
	a \in \left(\face(x,C)\right)^{\perp} \quad 
	\textrm{ satisfying } \quad\langle
	a,v\rangle <0,
	\end{equation}
	the point $\MM^*a$ lies in $(\cl \MM^*C^*)\setminus
	\MM^*C^*$, thereby certifying that $\MM^*C^*$ is not closed.
\end{thm}

\smallskip
\begin{rem}
	{\rm
		Following notation of Theorem~\ref{pataki-cl}, it is shown in \cite[Theorem 1.1, Corollary 3.1]{Pataki:07} that for any point $x\in\ri(C\cap \,\range
		\MM)$, we have equality
		\begin{equation*}
		(\tcone(x,C)\cap \range \MM) \setminus \dir(x,C) \, = \, (\tcone(x,C) \cap \range \MM)\setminus \linspan \face(x,C).
		\end{equation*}
		Hence for any point $x\in \ri(C\cap \,\range
		\MM)$ and any vector $v\in \left(\tcone(x,C) \cap \range \MM\right)\setminus \dir(x,C)$, there indeed exists some point $a$ satisfying \eqref{eq:av}. }
\end{rem}

\smallskip

The following sufficient condition for image closedness is now immediate. 
\smallskip
\begin{corollary}[Sufficient condition for image closedness]
	\label{cor:pataki-cl} \hfill \\
	Let $\MM:\E\rightarrow \Y$ be a linear transformation 
	between two Euclidean spaces $\E$ and $\Y$, and let $C\subseteq \Y$ be a
	nice, closed convex cone. If for some point $x\in\ri(C\cap \,\range
	\MM)$, the inclusion $\range(\MM) \subseteq \linspan \face(x,C)$ holds, then $\MM^*C^*$ is closed.
\end{corollary}
\begin{proof}
	Define $F:=\face(x,C)$ and note $\range(\MM) \subseteq \linspan F \subseteq \dir(x,C)$.
	We deduce
	\[
	\begin{array}{rcl}
	\tcone (x,C) \cap \range(\MM) 
	&\subseteq &
	\tcone (x,C) \cap \dir(x,C)
	
	=  \dir(x,C).
	\end{array}
	\]
	The result now follows from Theorem~\ref{pataki-cl}, since implication \eqref{eqn:unlin} holds.
\end{proof}

\smallskip

{\em Alternate proof of Theorem~\ref{thm:maincompl_fixed}.}
First, whenever $L=\emptyset$ one can easily verify the equation $\PP(\Sn_+)=\R^{E}$.
Hence the theorem holds trivially in this case.
Without loss of generality, we now permute the vertices $V$ so that we have $L=\{1,\ldots,r\}$ for some integer $r\geq 1$.
We will proceed by
applying Theorem~\ref{pataki-cl} with $\MM:=\PP^*$ and 
$C:=(\Sn_+)^*=\Sn_+$.
To this end, observe the equality 
\[
\Sn_+\cap \range \PP^*=\left\{\begin{bmatrix} A & 0 \\ 0 & 0 \end{bmatrix}: 
   A\in \Sr_+ \textrm{ and } A_{ij}=0 
       \textrm{ when } ij\notin E \right\}.
\]
Thus we obtain the inclusion
$$
X := \begin{bmatrix} I_r & 0 \\ 0 & 0 \end{bmatrix} \in
\ri (\Sn_+\cap \range \PP^*). 
$$
Observe 
$$\face(X,\Snp)= \left\{ \begin{bmatrix} A & 0 \\ 0 &
0 \end{bmatrix} : A\in \mathcal{S}^{r}_+ \right\}.$$
From~\cite[Lemma~3]{GPat:11},
we have the description
$$\tcone(X, \Sn_+) =
\left\{ \begin{bmatrix} A & B \\ B^T &
C \end{bmatrix} : C \in \mathcal{S}^{n-r}_+ \right\},$$
while on the other hand 
$$\dir(X, \Sn_+) =
 \left\{ \begin{bmatrix} A & B \\ B^T &
C \end{bmatrix} : C \in \mathcal{S}^{n-r}_+ ~ \textrm{ and } ~ \range
B^T\subseteq \range C \right\}.$$
Thus if the intersection 
$E \cap \bigl( \{1, \dots, r \} \times 
            \{r+1, \dots, n \} \bigr)$ is empty, then for 
any matrix 
\[
\begin{bmatrix} A & B \\ B^T &
C \end{bmatrix}\in \tcone(X,\Sn_+) \cap \range \PP^*,
\] 
we have $B=0$, and consequently this matrix lies in $\dir(X,\Sn_+)$. 
Using Theorem~\ref{pataki-cl}, we deduce that the image $\PP(\Sn_+)$ is closed.
Conversely, for any edge $i^*j^*\in E \cap \bigl( \{1, \dots,
r \} \times \{r+1, \dots, n \} \bigr)$, we can define the matrix 
$$V:=e_{ i^*}e_{j^{*}}^T+e_{j^*}e_{i^*}^T\in 
\left\{\tcone(X,\Sn_+)\setminus \dir(X,\Sn_+)\right\}\cap 
\range \PP^*,$$
where $e_{i^*}$ and $e_{j^*}$ denote the $i^*$'th and $j^*$'th unit vectors in $\R^n$. 
Theorem~\ref{pataki-cl} immediately implies that the image $\PP(\Sn_+)$ is not closed.
Moreover, in this case, define $A\in \Sn$ to be any matrix satisfying $A_{i^*j^*}< 0$ and $A_{ij}=0$ whenever $ij\in \{1,\ldots,r\}\times\{1,\ldots,r\}$. Then $A$ lies in $(\face(X,\Sn_+))^\perp$ and the inequality, $\langle A,V\rangle=2A_{i^*j^*} <0$, holds. Again appealing to Theorem~\ref{pataki-cl}, we deduce $\PP(A)\in (\cl \PP(\Sn_+))\setminus \PP(\Sn_+)$, as we had to show. Replacing $V$ by $-V$ shows that the same conclusion holds in the case $A_{i^*j^*}> 0$. This completes the proof.
$\square$

\smallskip

{\em Alternate proof of Theorem~\ref{thm:close_EDM_fixed}.}
As in the original proof, we can assume that $G$ is connected.
The proof proceeds by applying Corollary~\ref{cor:pataki-cl}.
To this end, in the notation of that corollary, we set $C:=\Snp$ and
$\MM=\mathcal{L}=\frac{1}{2}\KK^*\circ  \PP^*$. Clearly then we 
have the equality $\MM^*C^*=\PP(\mathcal{E}^n)$.
Define now the partial matrix $x \in\R^{E}$ with $x_{ij}=1$ for all $ij\in E$, 
and set $X :=\mathcal{L}(x)$.
We now claim that the inclusion
\begin{equation}\label{eq:Xrelint}
X \in \ri \left(\Snp \cap \range \mathcal{L} \right) \qquad \textrm{ holds}.
\end{equation}
To see this, observe that $X$ lies in the intersection 
$\Snp \cap \range \mathcal{L}$, since $X$ is a positively weighed Laplacian.
Now let $Y \in \Snp \cap \range \mathcal{L}$ be arbitrary, then 
$Y  = \mathcal{L}(y)$ for some partial matrix $y \in\R^E.$ Consider the matrices
$$
X \pm \epsilon(X-Y) \, = \, \mathcal{L}(x \pm \epsilon (x-y)).
$$
If $\epsilon >0$ is small, then $x \pm \epsilon (x-y)$ has all positive components, and so
$X \pm \epsilon(X-Y)$ is a positively weighed Laplacian, hence positive semidefinite.
This proves \eqref{eq:Xrelint}.
Now define $F=\face(X,\Snp)$. We claim that equation $\linspan F = \mathcal{S}_c$ holds. To see this, recall that the nullspace of $X$ is one-dimensional, being generated by $e$. Consequently $F$ has dimension $\frac{n(n-1)}{2}$. On the other hand $F$ is clearly contained in $\mathcal{S}_c$, a linear subspace of dimension $\frac{n(n-1)}{2}$. We deduce $\linspan F = \mathcal{S}_c$, as claimed.
The closure now follows from Corollary~\ref{cor:pataki-cl}.
$\square$

\section{Boundaries of projected sets \& facial reduction}
\label{sect:bndry}
To motivate the discussion, consider the general conic system
\begin{equation}\label{eqn:feas_reg}
F:=\{X\in C: \mathcal{M}(X)=b\},
\end{equation}
where $C$ is a proper (closed, with nonempty interior)
convex cone in an Euclidean space $\mathbb{E}$ and
$\mathcal{M}\colon\mathbb{E}\to\R^m$ is a linear mapping.
Classically we say that the Slater condition holds for this problem whenever there exists $X$ in the interior of 
$C$ 
satisfying the system $\mathcal{M}(X)=b$. 
In this section, we first relate properties of the image set $\MM(C)$ to
the {\em facial reduction algorithm} of Borwein-Wolkowicz
\cite{BW80,BW81}, and to more recent variants \cite{MR3063940,
MR3108446, perm}
and then specialize the discussion to the PSD and EDM completion problems we have been studying.

When strict feasibility fails, the {\em facial reduction} strategy aims
to embed the feasibility system in a Euclidean space of smallest
possible dimension. The starting point is the following elementary
geometric observation: exactly one of the following alternatives holds
\cite{BW80,BW81,MR3108446, perm}.
\newcounter{qcounter2}
\begin{list}{\arabic{qcounter2}.~}{\usecounter{qcounter2} \leftmargin=2em}
	\item The conic system in \eqref{eqn:feas_reg} is strictly feasible.
	\item\label{item:cirt} 
	There is a nonzero matrix $Y \in C^*$ so that the orthogonal complement $Y^{\perp}$ contains the affine space $\{X: \MM(X)=b\}$. 
\end{list}
The matrix $Y$ in the second alternative certifies that the entire
feasible region $F$ is contained in the slice $C \cap Y^{\perp}$.
Determining which alternative is valid is yet another system that needs to be solved, namely
find a vector $v$ satisfying the {\em auxiliary system}
$$0\neq \mathcal{M}^*v\in C^* \quad\textrm{ and }\quad \langle v,b\rangle =0,$$
and set $Y:=\mathcal{M}^*v$.
One can now form an ``equivalent'' feasible region to \eqref{eqn:feas_reg} by replacing $C$ with $C\cap Y^{\perp}$ and $\E$ with the linear span of $C\cap Y^{\perp}$.
One then repeats the procedure on this smaller problem, forming the
alternative, and so on and so forth until strict feasibility holds. The
number of steps for the procedure to terminate depends on the choices of
the exposing vectors $Y$. The minimal number of steps needed is the {\em
singularity degree} --- an intriguing measure of complexity
\cite{S98lmireport}. In general, the singularity degree is no greater
than $n-1$, and there are instances of semidefinite programming
 that require exactly $n-1$ facial reduction iterations \cite[Section 2.6]{lev_book}.

\smallskip

The following theorem provides an interesting perspective on facial reduction in terms of the image set $\MM(C)$. In essence, the minimal face of $\MM(C)$ containing $b$ immediately yields the minimal face of $C$ containing the feasible region $F$, that is in principle no auxiliary sequence of problems for determining $\face(F,C)$ is needed. The difficulty is that geometry of $\MM(C)$ is in general complex and so a simple description of $\face(b,\MM(C))$ is unavailable. The auxiliary problem in the facial reduction iteration instead tries to represent 
$\face(b,\MM(C))$ using some dual vector $v$ exposing a face of
$\MM(C)$ containing $b$. The singularity degree is then exactly one if,
and only if, the minimal face $\face(b,\MM(C))$ is exposed. To the best of our knowledge, this relationship to exposed faces has either been overlooked in the literature or forgotten. In particular, an immediate consequence is that whenever the image cone $\MM(C)$ is facially exposed, the feasibility problem \eqref{eqn:feas_reg} has singularity degree at most one for any right-hand-side vector $b$, for which the feasible region is nonempty.
\smallskip
\begin{thm}[Facial reduction and exposed faces]
\label{thm:face_red}
Consider a linear operator $\mathcal{M}\colon\E\to\Y$, between two Euclidean spaces $\E$ and $\Y$, and let $C\subset \E$ be a proper convex cone. 
Consider a nonempty feasible set 
\begin{equation}\label{eq:feas_sis}
F:=\{X\in C: \MM(X)=b\}
\end{equation}
for some point $b\in \Y$. Then a vector $v$ exposes a proper face of
$\MM(C)$ containing $b$ if, and only if, $v$ satisfies the auxiliary system
\begin{equation}\label{eq:aux}
0\neq\mathcal{M}^*v\in C^* \quad\textrm{ and }\quad \langle v,b\rangle =0.
\end{equation}
 For notational convenience, define $N:=\face(b, \MM(C))$. Then the following are true.
\begin{enumerate}	
\item\label{claim:1} We always have $C\cap\MM^{-1}N=\face(F, C)$.
\item\label{claim:3_in} For any vector $v\in\R^m$ the following equivalence holds:
$$v \textrm{ exposes } N \quad\Longleftrightarrow\quad \mathcal{M}^*v \textrm{ exposes } \face(F, C).$$
\end{enumerate}
Consequently whenever the Slater condition fails, the singularity degree
of the system \eqref{eq:feas_sis} is equal to one if, and only if, the minimal face $\face(b,\MM(C))$ is exposed.
\end{thm}
\begin{proof}
	First suppose that $v$ exposes a proper face of $\MM(C)$ containing $b$. Clearly we have $\langle v,b\rangle=0$. Observe moreover  
	$$\langle \MM^* v, X\rangle =\langle v, \MM(X)\rangle\geq 0, \qquad \textrm{ for any }X\in C,$$
and hence the inclusion $\MM^* v\in C^*$ holds. Finally, 
since $v$ exposes a proper face of $\MM(C)$, we deduce $v\notin (\range \MM)^{\perp}=\kernel \MM^*$. We conclude that $v$ satisfies the auxiliary system \eqref{eq:aux}. The converse follows along the same lines.	

We first prove claim \ref{claim:1}. To this end, we first verify that $C\cap\MM^{-1}N$ is a face of $C$. Observe for any $x,y\in C$ satisfying $\frac{1}{2}X+\frac{1}{2}Y\in C\cap\MM^{-1}N$, we have $\frac{1}{2}\MM(X)+\frac{1}{2}\MM(Y)\in N$. Since $N$ is a face of $\MM(C)$, we deduce $X,Y\in C\cap\MM^{-1}N$ as claimed. Now clearly $C\cap\MM^{-1}N$ contains $F$. It is easy to verify now the equality
$$N=\MM(C\cap \MM^{-1}N).$$
Appealing to \cite[Theorem~6.6]{con:70}, we deduce
$$\ri N= \MM(\ri (C\cap \MM^{-1}N))$$
Thus $b$ can be written as $\MM(X)$ for some $X\in \ri (C\cap \MM^{-1}N)$.
We deduce that the intersection $F\cap \ri (C\cap \MM^{-1}N)$ is nonempty. Appealing to~\cite[Proposition~2.2(ii)]{PatakiSVW:99}, we obtain the claimed equality $C\cap \MM^{-1}N=\face(F, C)$.


Finally we prove \ref{claim:3_in}.	To this end, suppose first that a vector $v$ exposes $N$. Then by what has already been proven $v$ satisfies the auxiliary system and therefore
$C\cap (\mathcal{M}^*v)^{\perp}$ is an exposed face of $C$ containing $F$. 
It is standard now to verify equality
\begin{equation}\label{eqn:equality_rep}
\MM(C\cap (\mathcal{\MM}^*v)^{\perp})=\MM(C)\cap v^{\perp}=N.
\end{equation}
Indeed, for any $a\in \MM(C)\cap v^{\perp}$, we may write $a=\MM(X)$ for some $X\in C$ and consequently deduce $\langle X,\mathcal{\MM}^*v\rangle =\langle a,v\rangle=0$. Conversely, for any $X\in C\cap (\mathcal{\MM}^*v)^{\perp}$, we have
$\langle\mathcal{\MM}(X),v\rangle=\langle X, \mathcal{\MM}^*v\rangle =0$, as claimed.

Combining equation \eqref{eqn:equality_rep} with ~\cite[Theorem~6.6]{con:70}, we deduce
$$\ri(N)=\MM(\ri (C\cap (\mathcal{M}^*v)^{\perp})).$$
Thus $b$ can be written as $\MM(X)$ for some $X\in \ri (C\cap (\mathcal{\MM}^*v)^{\perp})$.
We deduce that the intersection $F\cap \ri (C\cap (\mathcal{M}^*v)^{\perp})$ is nonempty. Appealing to~\cite[Proposition~2.2(ii)]{PatakiSVW:99}, we conclude that $C\cap (\mathcal{M}^*v)^{\perp}$ is the minimal face of $C$ containing $F$.

Now conversely suppose that $\MM^*v$ exposes  $\face(F, C)$. Then clearly $v$ exposes a face of $\MM(C)$ containing $b$. On the other hand, by claim \ref{claim:1}, we have 
$$C \cap \MM^{-1}N=\face(F, C)=C\cap (\MM^*v)^{\perp}.$$
Hence a point $\MM(X)$ with $X\in C$ lies in $\MM(C)\cap v^{\perp}$
if, and only if, it satisfies
$0=\langle v,\MM(X)\rangle=\langle \MM^*v,X\rangle$, which by the above equation is equivalent to the inclusion $\MM(X)\in N$. This completes the proof.
\end{proof}

\smallskip

The following example illustrates Theorem \ref{thm:face_red}.
\begin{example}[Singularity degree and facially exposed faces]
\label{ex:facialexp} {\rm Consider the cone $C := \psd{3}$. 
Define now the mapping $\MM: \sym{3} \rightarrow \rad{2}$ and the vector $b\in\R^2$ to be 
$$
\MM(X) = \begin{pmatrix} X_{11}\\ X_{33}\end{pmatrix} \quad \textrm{ and } \quad b = \begin{pmatrix}1 \\ 0\end{pmatrix}.
$$
Then the singularity degree of the system (\ref{eqn:feas_reg}), namely
$$\{X\in \psd{3}:X_{11}=1, X_{33}=0\},$$
 is one. Indeed $v = (0,1)^T$ is a solution to the auxiliary system, and hence Slater condition fails. On the other hand, the feasible matrix $X=I -e_3{e_3}^T$ shows that the maximal solution rank of the system is two.
This also follows immediately from Theorem~\ref{thm:face_red} since the image set
$\MM(\psd{3}) = \rad{2}_+$ is facially exposed. 

We next slightly change this example by adding a coordinate. Namely, 
define $\MM: \sym{3} \rightarrow \rad{3}$ and $b$ to be 
$$
\MM(X) = \begin{pmatrix} X_{11}\\ X_{33} \\X_{22} + X_{13} \end{pmatrix} \quad \textrm{ and } \quad b = \begin{pmatrix}1 \\ 0 \\ 0\end{pmatrix}.
$$
Observe that the singularity degree of the system \eqref{eqn:feas_reg} is at most two, since it admits a rank one feasible solution $X=e_1e_1^T$. On the other hand, one can see directly from Theorem~\ref{thm:face_red} that the singularity degree is exactly two.
Indeed, one easily checks
$$
\MM(\psd{3}) \, = \, \rad{3}_+ \cup \{ (x,y,z) \, | \, x \geq 0, \, y \geq 0, \, z \leq 0, \, xy \geq z^2 \, \},
$$
i.e., we obtain $\MM(\psd{3})$ by taking the union of $\rad{3}_+$ with 
a rotated copy of $\psd{2}$. The set $\MM(\psd{3})$ 
has a nonexposed face which contains $b$ in its relative interior -- this is easy to see 
by intersecting $\MM(\psd{3})$ with the hyperplane $x=1$ and graphing.
}
\end{example}

\smallskip

An interesting consequence of Theorem \ref{thm:face_red} above is that it is the lack of facial exposedness of the image set $\MM(C)$ that is responsible for a potentially large singularity degree and hence for serious numeric instability, i.e. weak H\"olderian error bounds \cite{S98lmireport}.

\smallskip

\begin{corollary}[H\"olderian error bounds and facial exposedness]\label{cor:Holder}
Consider a linear mapping $\mathcal{A}\colon\mathcal{S}^n\to\R^m$ having the property that $\mathcal{A}(\mathcal{S}^n_+)$ is facially exposed. 
For any vector $b\in\R^m$, define the affine space 
$$\mathcal{V}=\{X:\mathcal{A}(X)=b\}.$$
Then whenever the intersection $\mathcal{S}^n_+\cap \mathcal{V}$ is nonempty it admits a $\frac{1}{2}$-H\"{o}lder error bound: 
for any compact set $U$, there is a constant $c> 0$ so that  
$$\dist_{\mathcal{S}^n_+\cap \mathcal{V}}(X)\leq c\cdot\sqrt{\max\Big\{\dist_{\mathcal{S}^n_+}(X),\dist_{\mathcal{V}}(X)\Big\}}\qquad \textrm{ for all }x\in U.$$	
\end{corollary}
\begin{proof}
This follows immediately from Theorem~\ref{thm:face_red} and \cite{S98lmi}. 	
\end{proof}

\subsection{Facial reduction for completion problems}
For those problems with highly structured constraints, one can hope to solve the auxiliary problems directly.
For example, the following simple idea can be fruitful: fix a subset $I\subset \{1,\ldots,m\}$ and let $\mathcal{M}_I(X)$ and $b_I$, respectively, denote restrictions of $\mathcal{M}(X)$ and $b$ to coordinates indexed by $I$. Consider then the relaxation:
$$F_I:=\{X\in C: \mathcal{M}_I(X)=b_I\}.$$
If the index set $I$ is chosen so that the image $\mathcal{M}_I(C)$ is ``simple'', then we may find the minimal face $\face(F_I,C)$, as discussed above. Intersecting such faces for varying index sets $I$ may yield a drastic dimensional decrease. Moreover, observe that this preprocessing step is entirely parallelizable.

Interpreting this technique in the context of matrix completion problems, we recover the Krislock-Wolkowicz algorithm \cite{Nathan_Henry}.
Namely note that when $\mathcal{M}$ is simply the projection $\PP$ and we set $C=\Snp$ or $C=\EE^n$, we obtain the PSD and EDM completion problems,
\begin{equation*}
F:=\{X\in C: \PP(X)=a \}=\{X\in C: X_{ij}=a_{ij} \textrm{ for all }
ij\in E\},
\end{equation*}
where $a\in\R^E$ is a partial matrix.  
It is then natural to consider indices $I\subset E$ describing clique edges in the graph since then the images $\PP_I(C)$ are the smaller dimensional PSD and EDM cones, respectively ---  sets that are well understood. This algorithmic strategy becomes increasingly effective when the rank (for the PSD case) or the embedding dimension (for the EDM case) of the specified principal minors are all small. Moreover, we will show that under a chordality assumption, the minimal face of $C$ containing the feasible region is guaranteed to be discovered if all the maximal cliques were to be considered; see Theorems~\ref{thm:cliquesuffPSD} and \ref{thm:fin_term}. This, in part, explains why the EDM completion algorithm of \cite{Nathan_Henry} works so well. Understanding the geometry of $\PP_I(C)$ for a wider class of index sets $I$ would yield an even better preprocessing strategy.
We defer to \cite{Nathan_Henry} for extensive numerical results and implementation issues showing that the discussed algorithmic idea is extremely effective for EDM completions. 

In what follows,  by the term ``clique $\chi$ in $G$'' we will mean a collection of $k$ pairwise connected vertices of $G$. The symbol $|\chi|$ will indicate the cardinality of $\chi$ (i.e. the number of vertices) while $E(\chi)$ will denote the edge set in the subgraph induced by $G$ on $\chi$.
For a partial matrix $a\in\R^E$, the symbol $a_{\chi}$ will mean the restriction of $a$ to $E(\chi)$, whereas $\PP_{\chi}$ will be the projection of $\Ss^n$ onto $E(\chi)$. The symbol $\mathcal{S}^{\chi}$ will indicate the set of $|\chi|\times|\chi|$ symmetric matrices whose rows and columns are indexed by $\chi$. Similar notation will be reserved for  $\mathcal{S}^{\chi}_+$.
If $\chi$ is contained in $L$, then we may equivalently think of $a_{\chi}$ as a vector lying in $\R^{E(\chi)}$ or as a matrix lying in $\mathcal{S}^{\chi}$. Thus the adjoint $\PP^*_{\chi}$ assigns to a partial matrix $a_{\chi}\in \mathcal{S}^{\chi}$  an $n\times n$ matrix whose principal submatrix indexed by $\chi$ coincides with $a_{\chi}$ and whose all other entries are zero. 

\smallskip
\begin{thm}[Clique facial reduction for PSD completions]
\label{thm:PSD_face} 
Let $\chi \subseteq L$ be any $k$-clique in the graph $G$.
Let $a\in\R^{E}$ be a partial PSD matrix and define 
\begin{equation*}
\label{eq:minfacechi}
F_{\chi}:= 
\{X\in\Snp: X_{ij}=a_{ij} \textrm{ for all } ij\in E(\chi)\}.
\end{equation*}
Then for any matrix $v_{\chi}$ exposing $\face(a_{\chi},\Ss^\chi_+)$, the matrix
$$
\PP^*_{\chi}v_{\chi}
\quad \textrm{ exposes }\quad \face(F_{\chi},\Snp).
$$
\end{thm}
\begin{proof}
Simply apply Theorem~\ref{thm:face_red} with $C=\Snp$, $\mathcal{M}=\PP_{\chi}$, and $b=a_{\chi}$.
\end{proof}
\smallskip

Theorem~\ref{thm:PSD_face} is transparent and easy. Consequently it is natural to ask whether the minimal face of $\Snp$ containing the feasible region of a PSD completion problem can be found using solely faces arising from cliques, that is those faces described in Theorem~\ref{thm:PSD_face}. The answer is no in general: the following example exhibits a PSD completion problem that fails the Slater condition but for which all specified principal submatrices are definite, and hence all faces arising from Theorem~\ref{thm:PSD_face} are trivial.

\smallskip 
\begin{example}[Slater condition \& nonchordal graphs]
\label{ex:n4complbndr}{\hfill \\}
{\rm
Let $G=(V,E)$ be a cycle on four vertices with each vertex attached to a loop, that is $V=\{1,2,3,4\}$ and $E=\{12,23,34,14\}\cup\{11,22,33,44\}$.
Define the following PSD completion problems $C(\epsilon)$, parametrized by $\epsilon\geq 0$:
\begin{equation*}
C(\epsilon):\qquad\qquad
\begin{bmatrix}
1+\epsilon& 1 & ? & -1 \cr
1 & 1+\epsilon & 1 & ? \cr
? & 1 & 1 +\epsilon& 1 \cr
-1 & ? & 1 & 1 +\epsilon\cr
\end{bmatrix}.
\end{equation*}
Let $a(\epsilon)\in\R^E$ denote the corresponding partial matrices.
According to~\cite[Lemma 6]{GrJoSaWo:84} there is a unique positive semidefinite
matrix $A$ satisfying $A_{ij}=1, \forall |i-j|\leq 1$, namely the
matrix of all $1$'s. Hence the PSD completion problem $C(0)$ is infeasible, that is 
$a(0)$ lies outside of $\PP(\mathcal{S}^4_+)$. On the other hand, for all sufficiently large $\epsilon$, the partial matrices 
$a(\epsilon)$ do lie in $\PP(\mathcal{S}^4_+)$ due to the diagonal dominance. 
Taking into account that $\PP(\mathcal{S}^4_+)$ is closed (by Theorem~\ref{thm:maincompl_fixed}), we deduce that there exists
$\hat{\epsilon} >0$, so that $a(\hat{\epsilon})$ lies on the boundary of
$\PP(\mathcal{S}^4_+)$, that is the Slater condition \underline{fails} 
for the completion problem $C(\hat{\epsilon})$. On the other hand for all $\epsilon > 0$, the partial matrices 
$a(\epsilon)$ are clearly positive definite, and hence $a(\hat{\epsilon})$ is a partial PD matrix.  
In light of this observation, consider solving the semi-definite program:
\begin{equation}
\label{eq:primaleps}
\begin{array}{cll}
\min & \epsilon \\
\text{s.t.}     
          & 
\begin{bmatrix}
1+\epsilon & 1 & \alpha & -1 \cr
1 & 1+\epsilon & 1 & \beta \cr
\alpha & 1 & 1+\epsilon & 1 \cr
-1 & \beta & 1 & 1+\epsilon \cr
\end{bmatrix}\succeq 0\\
\end{array}
\end{equation}
Doing so, we deduce that $\hat{\epsilon}=\sqrt 2-1, \hat{\alpha}=\hat{\beta}=0$ is optimal. Formally, we can
verify this by finding the dual of \eqref{eq:primaleps} and
checking feasibility and complementary slackness for the primal-dual
optimal pair
$\widehat{X}$ and $\widehat{Z}$
\[
  \widehat{X} 
=\begin{bmatrix}
\sqrt 2 & 1 & 0 & -1 \cr
1 & \sqrt 2 & 1 & 0 \cr
0 & 1 & \sqrt 2 & 1 \cr
-1 & 0 & 1 & \sqrt 2 \cr
\end{bmatrix}, \qquad
  \widehat{Z}
=\frac 14 \begin{bmatrix}
1 & -\frac 1{\sqrt 2} & 0 & \frac 1{\sqrt 2} \cr
-\frac 1{\sqrt 2} & 1 & -\frac 1{\sqrt 2} & 0 \cr
0 & -\frac 1{\sqrt 2} & 1 & -\frac 1{\sqrt 2} \cr
\frac 1{\sqrt 2} & 0 & -\frac 1{\sqrt 2} & 1 \cr
\end{bmatrix}.
\]
}
\end{example}

Despite this pathological example, we now show that at least for chordal graphs, the minimal face of the PSD completion problem can be found solely from faces corresponding to cliques in the graph. We begin with the following simple lemma.
\smallskip

\begin{lem}[Maximal rank completions]\label{lem:max_rank_comp} 
Suppose without loss of generality $L=\{1,\ldots,r\}$ 
and let $G_L:=(L,E_L)$ be the graph induced on $L$ by $G. \,$  
Let $a\in\R^E$ be a partial matrix and $a_{E_L}$ the restriction of $a$ to $E_L.$ 
Suppose that $X_L \in \Srp$ is a maximum rank PSD completion of 
$a_{E_L}, $ and 
$$
X = \begin{bmatrix}
A & B\\
B^T & C
\end{bmatrix}
$$
is an arbitrary PSD completion of $a. \,$ 
Then 
$$
X_\mu := 
\begin{bmatrix}
X_L & B\\
B^T &  C+\mu I  
\end{bmatrix}
$$
is a maximal rank PSD completion of $a\in\R^E$ for all sufficiently large $\mu$.
\end{lem}
\begin{proof}
We \emph{construct} the maximal rank PSD completion from the arbitrary
PSD completion $X$ by moving from $A$ to $X_L$ and from $C$ to $C+\mu I$
while staying in the same minimal face for the completions. 
To this end, define the sets 
\[
\begin{array}{rcl}
F         & = & \left\{X\in \Snp : 
         X_{ij} = a_{ij}, \textrm{ for all } ij\in E\right\}, \\
F_L         & = & \left\{X\in \Srp : X_{ij} = a_{ij}, 
              \textrm{ for all } ij\in E_L\right\}, \\
\widehat{F} & = & \{X\in \Snp : X_{ij} = a_{ij}, \textrm{ for all } ij\in E_L\}.
\end{array}
\]
Then $X_L$ is a maximum rank PSD matrix in $F_L.$ 
Observe that the rank of any PSD matrix 
$
\begin{bmatrix}
P & Q\\
Q^T & R
\end{bmatrix}
$
is bounded by $\rank P+\rank R$. Consequently 
the rank of any PSD matrix in  $F$ and also in $\widehat{F}$ is bounded by 
$\rank X_L + (n-r), \,$ and the matrix 
$$
\bar{X}=\begin{bmatrix}
X_L & 0\\
0 & I
\end{bmatrix}
$$
has maximal rank in $\widehat{F}, \,$ 
i.e., 
\begin{equation}
\label{eq:rihatF}
\bar{X} \in \ri(\widehat{F}).
\end{equation}
Let $U$ be a matrix of eigenvectors of $X_L, \,$ with eigenvectors corresponding to 
$0$ eigenvalues coming first. Then 
$$
U^T X_L U = \begin{bmatrix} 0 & 0 \\ 0 & \Lambda \end{bmatrix},
$$
where $0\prec \Lambda  \in \Ss^k_+$ is a diagonal matrix with all positive diagonal elements.

Define 
$$
Q = \begin{bmatrix} U & 0 \\ 0 & I \end{bmatrix}.
$$
Let $X$ be as in the statement of the lemma; then clearly 
$X \in \widehat{F}$ and we deduce using \eqref{eq:rihatF} that
\begin{eqnarray} \label{ttxt}
\bar{X} \pm \epsilon (\bar{X} - X)  \in  \Snp \, & \Leftrightarrow & Q^T
\bar{X}Q  \pm \epsilon Q^T(\bar{X} - X)Q \in  \Snp,
\end{eqnarray}
for some small $\epsilon > 0.$ 
We now have 
\[
\begin{array}{rclcl} 
Q^T \bar{X}Q & = & \begin{bmatrix} U^T X_L U & 0 \\ 0 & I \end{bmatrix} & = & \begin{bmatrix} 0 & 0 & 0 \\ 0 & \Lambda & 0 \\ 0 & 0 & I \end{bmatrix}, \\
Q^T X Q & = & \begin{bmatrix} U^T A U & U^T B \\ B^T U & C \end{bmatrix} & = & \begin{bmatrix} V_{11} & V_{12} & V_{13} \\ V_{12}^T & V_{22} & V_{23} \\ V_{13}^T & V_{23}^T & V_{33} \end{bmatrix},
\end{array}
\]
where $V_{11} \in \mathcal{S}^{r-k}, \, V_{22} \in \mathcal{S}^k, \, V_{33} \in \mathcal{S}^{n-r}$. From 
(\ref{ttxt}) we deduce $V_{11} = 0, \, V_{12} = 0, \, V_{13} = 0.$ 
Therefore 
\[
\begin{array}{rcl} 
Q^T X_\mu Q & = & \begin{bmatrix} U^T X_L U & U^T B \\ B^T U & \mu I + C \end{bmatrix} = \begin{bmatrix} 0 & 0 & 0 \\ 0 & \Lambda & V_{23} \\ 0 & V_{23}^T & \mu I + C \end{bmatrix}.
\end{array}
\]
By the Schur complement condition  for positive semidefiniteness we have that for sufficiently large $\mu$ 
the matrix $X_\mu$ is PSD, and
$\rank X_\mu = \rank X_L + (n-r)$; hence it is a maximal rank PSD matrix in $F.$
\end{proof}

\smallskip
\begin{thm}[Finding the minimal face on chordal graphs]
\label{thm:cliquesuffPSD}
Suppose that the graph induced by $G$ on $L$ is chordal. Consider a partial PSD matrix $a\in \R^{E}$ and the region
\[
F=\{X\in \Snp : X_{ij} = a_{ij} \textrm{ for all } ij\in E\}.
\] 
Then the equality $$\face(F, \Snp)=\bigcap_{\chi\in \Theta} \face(F_{\chi},\Snp)\qquad \textrm{ holds},$$
where $\Theta$ denotes the set of all maximal cliques in the restriction of $G$ to $L$, and for each $\chi \in \Theta$ we define the relaxation
$$
F_{\chi}:=\{X\in \Snp : X_{ij} = a_{ij} \textrm{ for all } ij\in E(\chi) \}.
$$
\end{thm}
\begin{proof}
For brevity, set 
$$
H = \bigcap_{\chi\in \Theta} \face(F_{\chi},\Snp).
$$
We first prove the theorem under the assumption that $L$ is disconnected from $L^c$. 
To this end, for each clique $\chi\in\Theta$, let $v_{\chi}\in \Ss^{\chi}_+$ denote the exposing vector of $\face(a_{\chi}$, $\Ss^{\chi}_+)$. Then by Theorem~\ref{thm:PSD_face}, we have
$$\face(F_{\chi},\Snp)=\Snp\cap (\PP^*_{\chi}v_{\chi})^{\perp}.$$
It is straightforward to see that $\PP^*_{\chi}v_{\chi}$ is simply the $n\times n$ matrix whose principal submatrix indexed by $\chi$ coincides with $v_{\chi}$ and whose all other entries are zero. Letting $Y[\chi]$ denote the principal submatrix indexed by $\chi$ of any matrix $Y\in\Snp$, we successively deduce 
\begin{align*}
\PP(H) &=\PP 
\Big(\{Y\succeq 0: Y[\chi]\in v^{\perp}_{\chi} \quad\forall \chi\in\Theta\} \Big)\\
&=\PP(\Snp)\cap \{b\in\R^E: b_\chi\in v^{\perp}_{\chi} \quad\forall \chi\in\Theta\}.
\end{align*}
On the other hand, since the restriction of $G$ to $L$ is chordal and $L$ is disconnected from $L^c$, Theorem~\ref{thm:psd_comp} implies that $G$ is PSD completable. Hence we have 
the representation $\PP(\Snp)=\{b\in\R^E: b_{\chi}\in\mathcal{S}^{\chi}_+ \quad\forall \chi\in\Theta\}$. Combining this with the equations above, we obtain
\begin{align*}
\PP(H)&= \{b\in\R^E: b_\chi\in \mathcal{S}^{\chi}_+ \cap v^{\perp}_{\chi} \quad\forall \chi\in\Theta\}  \\
&= \{b\in\R^E: b_{\chi}\in \face(a_{\chi}, \Ss^{\chi}_+) \quad\forall \chi\in\Theta\}\\
&= \bigcap_{\chi\in \Theta}  \{b\in\R^E: b_{\chi}\in \face(a_{\chi}, \Ss^{\chi}_+) \},
\end{align*}
Clearly $a$ lies in the relative interior of each set $\{b\in\R^E: b_{\chi}\in \face(a_{\chi}, \Ss^{\chi}_+) \}$. Using \cite[Theorems~6.5,6.6]{con:70}, we deduce
$$a\in\ri  \PP(H) =\PP(\ri H).$$ 
Thus the intersection $F\cap \ri H$ is nonempty. 
Taking into account that $F$ is contained in $H, \,$ 
and appealing to~\cite[Proposition~2.2(ii)]{PatakiSVW:99}, we conclude that $H$ 
is the minimal face of $\Snp$ containing $F$, as claimed.

We now prove the theorem in full generality, that is when there may exist an edge joining $L$ and $L^c$. To this end,
let $\widehat{G}_L = (V, E_L)$ be 
the graph obtained from $G$ by deleting all edges adjacent to $L^c. \,$ 
Clearly, $L$ and $L^c$ are disconnected in $\widehat{G}_L.$ 
Applying the special case of the theorem that we have just proved, we deduce that in terms of the set  
\[
\begin{array}{rcl}
\widehat{F} & = & \{X\in \Snp : X_{ij} = a_{ij} \textrm{ for all } ij\in E_L\},
\end{array}
\]
we have 
$$\face(\widehat{F}, \Snp)=H.$$
The $X_\mu$ matrix of Lemma~\ref{lem:max_rank_comp} is a maximum rank PSD matrix in
$F, \,$ and also in $\widehat{F}.$ Since $F \subseteq \widehat{F}, \,$ 
we deduce $\face(F,\Snp)=\face(\widehat{F},\Snp)$, and this completes the proof.
\end{proof}


\smallskip


\smallskip
\begin{remark}[Finding maximal cliques on chordal graphs]
{\rm 
In light of the theorem above, it noteworthy that finding maximal cliques on chordal graphs is polynomially solvable; see e.g. \cite{Laurent:00} or more generally \cite{poly_solve_chord}.
}
\end{remark}

\smallskip
The following is an immediate consequence.
\begin{corollary}[Singularity degree of chordal completions]\label{cor:sing_psd}{\hfill \\ }
	If the restriction of $G$ to $L$ is chordal, then the PSD completion problem has singularity degree at most one.  	
\end{corollary}
\begin{proof}
In the notation of Theorem~\ref{thm:cliquesuffPSD}, the sum
$Y:=\sum_{\chi\in\Theta} \PP^*_{\chi}v_{\chi}$ exposes the minimal face $\face(F,\mathcal{S}^+_n)$. If the Slater condition fails, then $Y$ is feasible for the first auxiliary problem in the facial reduction sequence. 
\end{proof}

\smallskip
\begin{example}[Finding the minimal face on chordal graphs]
{\rm
Let $\Omega$ consist of all matrices $X\in\mathcal{S}^4_+$ solving the PSD completion problem 
$$\begin{bmatrix}[r]
1 & 1 & ? & ? \\
1 & 1 & 1 & ? \\
? & 1 & 1 & -1 \\
? & ? & -1 & 2 \\
\end{bmatrix}
.$$
There are three nontrivial cliques in the graph. Observe that the minimal face of $\mathcal{S}^2_+$ containing the matrix 
$$\begin{bmatrix}
1 & 1\\
1 & 1\\
\end{bmatrix}
= \begin{bmatrix}[r]
-\frac{1}{2} & \frac{1}{2}\\
\frac{1}{2} & \frac{1}{2}\\
\end{bmatrix} \begin{bmatrix}[r]
0 & 0\\
0 & 4\\
\end{bmatrix} \begin{bmatrix}[r]
-\frac{1}{2} & \frac{1}{2}\\
\frac{1}{2} & \frac{1}{2}\\
\end{bmatrix} 
$$
is exposed by  
$$\begin{bmatrix}[r]
-\frac{1}{2} & \frac{1}{2}\\
\frac{1}{2} & \frac{1}{2}\\
\end{bmatrix} \begin{bmatrix}[r]
4 & 0\\
0 & 0\\
\end{bmatrix} \begin{bmatrix}[r]
-\frac{1}{2} & \frac{1}{2}\\
\frac{1}{2} & \frac{1}{2}\\
\end{bmatrix} 
= \begin{bmatrix}[r]
1 & -1\\
-1 & 1\\
\end{bmatrix}
.$$
Moreover, the matrix $\begin{bmatrix}
	1 & -1\\
	-1 & \,\,\,2\\
\end{bmatrix}$
is definite and hence the minimal face of $\mathcal{S}^2_+$ containing this matrix is exposed by the all-zero matrix.

Classically, an intersection of exposed faces is exposed by the sum of their exposing vectors. Using Theorem~\ref{thm:cliquesuffPSD}, we deduce that the minimal face of $\mathcal{S}^4_+$ containing $\Omega$ is the one exposed by the sum
$$
\begin{bmatrix}[r]
1 & -1 & 0 &0\\
-1 & 1 & 0 &0\\
0 & 0 & 0 &0\\
0 & 0 & 0 &0\\
\end{bmatrix}
+\begin{bmatrix}[r]
0 & 0 & 0 &0\\
0 & 1 & -1 & 0 \\
0 & -1 & 1 & 0 \\
0 & 0 & 0 &0\\
\end{bmatrix}=\begin{bmatrix}[r]
1 & -1 & 0 & 0 \\
-1 & 2 & -1 & 0 \\
0 & -1 & 1 &0\\
0 & 0 & 0 &0\\
\end{bmatrix}.
$$
Diagonalizing this matrix, we obtain
$$\face(\Omega, \mathcal{S}^4_+)= \begin{bmatrix}[r]
0 & 1  \\
0 & 1  \\
0 & 1  \\
3 & 0  \\
\end{bmatrix} 
\mathcal{S}^2_+
\begin{bmatrix}[r]
0 & 1 \\
0 & 1  \\
0 & 1  \\
3 & 0  \\
\end{bmatrix}^T.
$$
}
\end{example}

\smallskip

We now turn to an analogous development for the EDM completion problem. To this end, recall from \eqref{eq:isomorph} that the mapping $\KK\colon\Sn\to\Sn$ restricts to an isomorphism $\KK\colon \Ss_c\to \mathcal{S}_H$ carrying $\Ss_c\cap \Snp$
 onto $\mathcal{E}^n$. Moreover, it turns out that the Moore-Penrose pseudoinverse $\KK^{\dag}$ restricts to the inverse of this isomorphism $\KK^{\dag}\colon \mathcal{S}_H\to\mathcal{S}_c$. As a result, it is convenient to study the faces of $\EE^n$ using the faces of $\Ss_c\cap \Snp$. This is elucidated by the following standard result.
\smallskip
\begin{lemma}[Faces under isomorphism]\label{lem:face_iso} {\hfill \\ }
Consider a linear isomorphism $\mathcal{M}\colon\E\to\Y$ between linear spaces $\E$ and $\Y$, and let $C\subset \E$ be a closed convex cone. Then the following are true
\begin{enumerate}
\item $F\unlhd C  \quad \Longleftrightarrow \quad \mathcal{M}F\unlhd \mathcal{M}C$.    
\item $(\MM C)^{*}=(\MM^{-1})^*C^*.$
\item For any face $F\unlhd C$, we have $(\MM F)^{\triangle}= (\MM^{-1})^*F^{\triangle}$.
\end{enumerate}
\end{lemma}
\smallskip  
 
In turn, it is easy to see that $\Ss_c\cap\Snp$ is a face of $\Snp$ isomorphic to $\Ss^{n-1}_+$. More specifically for any $n\times n$ orthogonal matrix $ \begin{bmatrix}
\frac{1}{\sqrt n}e & U  \cr
\end{bmatrix}$,
we have the representation
\begin{equation*}
\mathcal{S}_c\cap \Snp= U\mathcal{S}^{n-1}_+U
\end{equation*}
Consequently, with respect to the ambient space $\mathcal{S}_c$, the cone $\mathcal{S}_c\cap \Snp$ is self-dual and for any face $F\unlhd \mathcal{S}^{n-1}_+$ we have
$$ UFU^T\unlhd \mathcal{S}_c\cap \Snp \qquad \textrm{ and } \qquad (UFU^T)^{\triangle} =UF^{\triangle}U^T.$$

As a result of these observations, we make the following important convention:
the \textdef{ambient spaces} of $\Ss_c\cap \Snp$  and of $\EE^n$ will always be taken as $\Ss_c$ and $\Ss_H$, respectively. Thus the facial conjugacy operations of these two cones will always be taken with respect to these ambient spaces and {\em not} with respect to the entire $\Ss^n$. 

Given a clique $\chi$ in $G$, we let $\mathcal{E}^{\chi}$ denote the set of $|\chi|\times |\chi|$ Euclidean distance matrices indexed by $\chi$. In what follows, given a partial matrix $a\in\R^E$, the restriction $a_{\chi}$ can then be thought of either as a vector in $\R^{E(\chi)}$  or as a hollow matrix in $\mathcal{S}^{\chi}$.
We will also use the symbol $\KK_{\chi}\colon \mathcal{S}^{\chi}\to\mathcal{S}^{\chi}$ to indicate the mapping $\KK$ acting on $\mathcal{S}^{\chi}$.

\smallskip
\begin{thm}[Clique facial reduction for EDM completions]\label{thm:EDM_face}
Let $\chi$ be any $k$-clique in the graph $G$.
Let $a\in\R^{E}$ be a partial Euclidean distance matrix and define 
$$
\begin{array}{rcl}
F_{\chi} &:= &
\{X\in \Snp\cap \mathcal{S}_c:  [\KK(X)]_{ij}=a_{ij} \textrm{ for
all } ij\in E(\chi)\}
\end{array}
$$
Then for any matrix $v_\chi$ exposing $\face\big(\KK^{\dag}(a_{\chi}),\Ss^{\chi}_+\cap \mathcal{S}_c\big)$, the matrix
$$
\PP^{*}_{\chi}v_{\chi}\quad \textrm{ exposes }\quad \face(F,\Snp \cap \mathcal{S}_c).
$$
\end{thm}
\begin{proof}
The proof proceeds by applying Theorem~\ref{thm:face_red} with
\[
C:= \Snp \cap \mathcal{S}_c,\quad \MM:= P_{\chi}\circ \KK, \quad b:=a_\chi.
\]
To this end, first observe
$\MM(C)=(P_{\chi}\circ \KK)(\Snp\cap\mathcal{S}_c)=\mathcal{E}^{\chi}$.
By Lemma~\ref{lem:face_iso}, the matrix $\KK^{\dag*}_{\chi}(v_{\chi})$ exposes $\face(a_{\chi}, \mathcal{E}^{\chi})$.
Thus the minimal face of $\Snp\cap\mathcal{S}_c$ containing $F$ is the one exposed by the matrix 
$$(P_{\chi}\circ\KK)^*(\KK^{\dag*}_{\chi}(v_{\chi}))=\KK^*P_{\chi}^*\KK_{\chi}^{\dag*}(v_{\chi})=P_{\chi}^* \KK_{\chi}^*\KK^{\dag*}_{\chi}(v_{\chi})=P_{\chi}^*v_{\chi}.
$$
The result follows.
\end{proof}

\smallskip
\begin{thm}[Clique facial reduction for EDM is sufficient]\label{thm:fin_term}
\label{thm:EDM_facesuff}
Suppose that $G$ is chordal, and consider a partial Euclidean distance matrix $a\in \R^{E}$ and the region
\[
F:=\{X\in \Ss_c\cap \Snp : [\KK(X)]_{ij} = a_{ij} \textrm{ for all } ij\in E\}.
\]
Let $\Theta$ denote the set of all maximal cliques in $G$, and for each $\chi \in \Theta$ define
\[
F_{\chi}:=\{X\in \Ss_c\cap \Snp : [\KK(X)]_{ij} = a_{ij} \textrm{ for all } ij\in E(\chi)\}.
\]
Then the equality $$\face(F, \Ss_c\cap\Snp)=\bigcap_{\chi\in \Theta} \face(F_{\chi},\Ss_c\cap\Snp)\qquad \textrm{ holds}.$$ 
\end{thm}
\begin{proof}
The proof follows entirely along the same lines as the first part of the proof of Theorem~\ref{thm:cliquesuffPSD}. We omit the details for the sake of brevity.
\end{proof}

\smallskip

\begin{corollary}[Singularity degree of chordal completions]\label{cor:sing}
If the graph $G=(V,E)$ is chordal, then the EDM completion problem has singularity degree at most one, when feasible.  	
\end{corollary}
\smallskip

\section*{Conclusion}
In this manuscript, we considered properties of the coordinate shadows of the PSD and EDM cones: $\PP(\mathcal{S}^n_+)$ and $\PP(\mathcal{E})$. We characterized when these sets are closed, related their boundary structure to a facial reduction algorithm of \cite{Nathan_Henry}, and explained that the nonexposed faces of these sets directly impact the complexity of the facial reduction algorithm.
In particular, under a chordality assumption, the ``minimal face'' of the feasible region admits a combinatorial description, the singularity degree of the completion problems are at most one, and the coordinate shadows, $\PP(\mathcal{S}^n_+)$ and $\PP(\mathcal{E})$, are facially exposed. 
This brings up an intriguing follow-up research agenda:
\bigskip

\begin{changemargin}{0.5cm}{0.5cm} Classify graphs $G$ for which the images $\PP(\Snp)$ and $\PP(\EE)$ are facially exposed, or equivalently those for which the corresponding completion formulations have singularity degree at most one irrespective of the known matrix entries.
\end{changemargin}

\bigskip
\section*{Acknowledgments}
\label{sect:ack}
We thank the anonymous referees for their insightful comments. We are indebted to one of the referees for pointing out the elementary proofs of Theorems~\ref{thm:maincompl_fixed} and \ref{thm:close_EDM_fixed} presented in the beginning of Section~\ref{sect:closedness}.


\newpage
\bibliographystyle{plain}
\bibliography{master,psd,edm}

\end{document}